\documentclass[12pt,reqno, english]{amsart}
\usepackage[T1]{fontenc}
\usepackage[latin9]{inputenc}
\usepackage[a4paper]{geometry}
\usepackage{amsthm}
\usepackage{amssymb}

\makeatletter
\numberwithin{equation}{section}
\numberwithin{figure}{section}

\newtheorem{theorem}{Theorem}[section]
\newtheorem{remark}{Remark}[section]
\newtheorem{corollary}{Corollary}[section]

\usepackage{color}
\usepackage{hyperref}
\hypersetup{
   colorlinks = true,%
   citecolor = blue,
   filecolor=red,%
   linkcolor = [rgb]{0.65,0.0,0.0},%
   anchorcolor = red,
   pagecolor = red,
   urlcolor= [rgb]{0.65,0.0,0.0}
   linktocpage=true,
   pdfpagelabels=true,
   bookmarksnumbered=true,
}

\makeatother

\usepackage{babel}

\title[Compact manifolds vs. Sobolev-type inequalities]
{Topological rigidity of compact manifolds supporting Sobolev-type inequalities}


\author{Csaba Farkas}
\address{Department of Mathematics and Informatics, Sapientia University, Tg. Mure\c s, Romania 
}
\email{farkas.csaba2008@gmail.com; farkascs@ms.sapientia.ro}

\author{Alexandru Krist\'aly}
\address{\noindent Department of Economics, Babe\c s-Bolyai University, Cluj-Napoca, Romania, \newline \noindent Institute of Applied Mathematics,
	\'Obuda University, 1034 Budapest, Hungary}
\email{kristaly.alexandru@nik.uni-obuda.hu; alex.kristaly@econ.ubbcluj.ro}

\author{\'Agnes Mester}
\address{Institute of Applied Mathematics,
	\'Obuda University, 1034 Budapest, Hungary}
\email{mester.agnes@yahoo.com}

\begin{document}
\maketitle

\begin{abstract}
	Let $(M,g)$ be an $n$-dimensional $(n\geq 3)$ compact Riemannian manifold  with Ric$_{(M,g)}\geq (n-1)g$. If $(M,g)$ supports an  AB-type critical Sobolev inequality with Sobolev constants close to the optimal ones corresponding to the standard unit sphere $(\mathbb S^n,g_0)$, we prove that $(M,g)$ is  topologically close to $(\mathbb S^n,g_0)$. Moreover, the Sobolev constants on $(M,g)$ are precisely the optimal constants on the sphere $(\mathbb S^n,g_0)$ if and only if $(M,g)$ is isometric to $(\mathbb S^n,g_0)$; in particular, the latter result answers a question of V.H. Nguyen.  
\end{abstract}

\section{Introduction}

Let $(M,g)$ be a smooth compact $n$-dimensional Riemannian manifold, $n\geq 3.$ The general theory of Sobolev inequalities shows that there exist $A>0$ and $B>0$ such that 
\begin{equation}\label{AB}
 \left(\int_{M}|u|^{\frac{2n}{n-2}}{\rm d}v_{g}\right)^{\frac{n-2}{n}}\leq  A\int_{M}|\nabla_{g}u|^{2}{\rm d}v_{g}+B\int_{M}u^{2}{\rm d}v_{g,}\ \forall u\in H_1^2(M).
\end{equation}
In fact, problem (\ref{AB}) is a part of the famous AB-program initiated by Aubin \cite{Aubin} concerning the optimality of the constants $A$ and $B$; for a systematic presentation of this topic, see the monograph of Hebey \cite[Chapters 4 \& 5]{Hebey}.  
In particular, one can prove the existence of $B>0$ such that (\ref{AB}) holds with $A=A_0=\frac{4}{n(n-2)}\omega_{n}^{-\frac{2}{n}}$, cf. \cite[Theorem 4.6]{Hebey}, the latter value being the optimal Talenti constant in the Sobolev embedding $H_1^2{(\mathbb R^n)}\hookrightarrow L^{2^*}(\mathbb R^n),$ $n\geq 3,$ where $2^*=2n/(n-2).$ Hereafter, $\omega_n={\rm Vol}_{g_0}(\mathbb S^n)$ denotes the volume of the standard unit sphere $(\mathbb S^n,g_0)$.  If $u\equiv 1$ in (\ref{AB}), then we have $B\geq {\rm Vol}_{g}(M)^{-\frac{2}{n}}$, where Vol$_g(S)$ denotes the volume of $S\subset M$ in $(M,g)$. Moreover, if $n\geq 4$ then the validity of (\ref{AB}) with $A=A_0=\frac{4}{n(n-2)}\omega_{n}^{-\frac{2}{n}}$ implies  $$B\geq \frac{1}{n(n-1)}\omega_{n}^{-\frac{2}{n}}\max_M {\rm Scal}_{(M,g)},$$ where ${\rm Scal}_{(M,g)}$ is the scalar curvature of $(M,g),$ cf. \cite[Proposition 5.1]{Hebey}. 

In the model case when $(M,g)=(\mathbb S^n,g_0)$  is the standard unit sphere of $\mathbb R^{n+1}$, Aubin \cite{Aubin}  proved that the optimal values of $A$ and $B$ in (\ref{AB}) are 
\begin{equation}\label{AnullBnull}
A_0=\frac{4}{n(n-2)}\omega_{n}^{-\frac{2}{n}} \ {\rm and}\ B_0=\omega_{n}^{-\frac{2}{n}},
\end{equation}
respectively; 
moreover, for every $\lambda>1$, the  function  $u_\lambda(x)=(\lambda-\cos d_0(x))^{1-\frac{2}{n}}$, $x\in \mathbb S^n,$ is extremal in (\ref{AB}), see also \cite[Theorem 5.1]{Hebey}. Hereafter,  $d_0(x)=d_{\mathbb S^n}(y_{0},x)$, $x\in \mathbb S^n,$ where $d_{\mathbb S^n}$ denotes the standard metric on $(\mathbb S^n,g_0)$ and the element $y_0\in \mathbb S^n$ is arbitrarily fixed. 
Note however that on the quotients $M=\mathbb S^1(r)\times \mathbb S^2$  of $\mathbb S^{3}$ endowed with its natural metric $g$ (with $r>0$ sufficiently small) inequality (\ref{AB}) is not valid for $A=A_0$ and $B={\rm Vol}_{g}(M)^{-\frac{2}{n}},$ see \cite[Proposition 5.7]{Hebey}.

Let $B_M(x,\rho)$ and $B_{\mathbb S^n}(y,\rho)$ be the open geodesic balls with radius $\rho>0$ and centers in $x\in M$ and $y\in \mathbb S^n$ in  $(M,g)$ and $(\mathbb{S}^{n},g_0)$, respectively.  

Our main result reads as follows:

\begin{theorem}\label{fotetel}
Let $(M,g)$ be an $n$-dimensional $(n\geq 3)$ compact Riemannian manifold with Ricci curvature 
$\mathrm{Ric}_{(M,g)}\geq(n-1)g$ and assume  that the Sobolev inequality $(\ref{AB})$ holds on $(M,g)$ with some constants $A,B>0$.  Then the following assertions hold: 
\begin{itemize}
	\item[(i)] $A\geq A_0$ and $B\geq B_0$, where $A_0$ and $B_0$ are  from $(\ref{AnullBnull});$
	\item[(ii)] there exists $x_0\in M$ such that for every $y_0\in \mathbb S^n$ and $\rho\in [0,\pi]$,  
	\begin{equation}\label{foegyenlotlenseg}	
	{\rm Vol}_g(B_M(x_0,\rho))\geq \min\left\{\frac{A_0}{A},\frac{B_0}{B}\right\}^\frac{n}{2} 
	{\rm Vol}_{g_0}(B_{\mathbb S^n}(y_0,\rho)).
	\end{equation}
\end{itemize}
\end{theorem}

\begin{remark}\rm \label{remarkelso} Note that (\ref{foegyenlotlenseg}) is valid on the whole $[0,\infty)$. Indeed, since the Ricci curvature on $(M,g)$ verifies Ric$_{(M,g)}\geq (n-1)g$, due to Bonnet-Myers theorem, the diameter $D_M:={\rm diam}(M)$  of $M$  is bounded  from above by $\pi$; accordingly, for every $\rho\geq \pi$ one has  $B_M({x_0},\rho)=M$ and $B_{\mathbb S^n}({y_0},\rho)=\mathbb S^n$, thus  (\ref{foegyenlotlenseg}) can be extended beyond $\pi$.
\end{remark}

 Perelman \cite{Perelman} states that for every $n\geq 2$ there exists $\delta_n\in [0,1)$ such that if the $n$-dimensional compact Riemannian manifold $(M,g)$ with Ricci curvature $\mathrm{Ric}_{(M,g)}\geq(n-1)g$ verifies ${\rm Vol}_g(M)\geq (1-\delta_n){\rm Vol}_{g_0}(\mathbb S^n),$ then 
$M$ is homeomorphic to $\mathbb S^n$; this result has been improved by Cheeger and Colding \cite[Theorem A.1.10]{CC-97} by replacing homeomorphic to diffeomorphic. The latter result, the equality case in Bishop-Gromov inequality and Theorem \ref{fotetel} imply the following topological rigidity for compact manifolds: 

\begin{corollary}\label{kovetkezmeny} Under the same assumptions as in Theorem {\rm \ref{fotetel}}, if 
	$$\max\left\{\frac{A}{A_0},\frac{B}{B_0}\right\}\leq (1-\delta_n)^{-\frac{2}{n}},$$
then $(M,g)$ is diffeomorphic to $(\mathbb S^n,g_0)$. Moreover, $A=A_0$ and $B=B_0$ if and only if  $(M,g)$ is isometric to $(\mathbb{S}^{n},g_0)$. 
	\end{corollary}

\begin{remark}\rm 
 The statement of Corollary \ref{kovetkezmeny} is in the spirit of the results of Ledoux \cite{Ledoux} and do Carmo and Xia \cite{doC-Xia}. In these works certain Sobolev inequalities are considered on {\it non}-compact Riemannian manifolds with non-negative Ricci curvature, and the Riemannian manifold is isometric to the Euclidean space with the same dimension if and only if the Sobolev constants are precisely the Euclidean optimal constants.  Further results in this direction can be found in the papers by   Krist\'aly \cite{Kristaly-Calc-V, Kristaly-Potential} and Krist\'aly and Ohta \cite{Kristaly-Ohta}.   Theorem \ref{fotetel} and Corollary \ref{kovetkezmeny} seem to be the first contributions within this topic in the setting of compact Riemannian manifolds, answering also a question of Nguyen \cite{Nguyen}. 
 
\end{remark}

\section{Proofs}

\begin{proof}[Proof of Theorem \ref{fotetel}]

(i) The validity of the Sobolev inequality $(\ref{AB})$ on $(M,g)$ and a similar argument as in Hebey \cite[Proposition 4.2]{Hebey}  imply that $A\geq A_0$. 

By Remark \ref{remarkelso}, we have $D_M:=$diam$(M)\leq \pi$.
Since  Ric$_{(M,g)}\geq (n-1)g$, by the Bishop-Gromov comparison principle we have that for every $x_0\in M$ and $y_0\in \mathbb S^n$, the function  $\rho\mapsto \frac{{\rm Vol}_g(B_M({x_0},\rho))}{{\rm Vol}_{g_0}(B_{\mathbb S^n}({y_0},\rho))}$ is non-increasing on $(0,\infty)$; in particular, 
we have 
\begin{equation}\label{Bishop-Gromov}
1\geq \frac{{\rm Vol}_g(B_M({x_0},\rho))}{{\rm Vol}_{g_0}(B_{\mathbb S^n}({y_0},\rho))}\geq \frac{{\rm Vol}_g(B_M({x_0},\pi))}{{\rm Vol}_{g_0}(B_{\mathbb S^n}({y_0},\pi))}=\frac{{\rm Vol}_g(M)}{{\rm Vol}_{g_0}(\mathbb S^n)},\ \forall \rho\in [0,\pi].
\end{equation} 
Now, choosing $u\equiv 1$ in (\ref{AB}), it follows that 
$$B\geq {\rm Vol}_g(M)^{-\frac{2}{n}}\geq {\rm Vol}_{g_0}(\mathbb S^n)^{-\frac{2}{n}}=\omega_n^{-\frac{2}{n}}=B_0.$$

(ii) If $D_M= \pi$, we have nothing to prove. Indeed, in this case  $(M,g)$ is isometric to $(\mathbb S^n,g_0)$, see Cheng \cite{Cheng} and Shiohama \cite{Shiohama}, i.e.,  ${\rm Vol}_g(M)={\rm Vol}_{g_0}(\mathbb S^n)$ and (\ref{Bishop-Gromov}) implies at once relation (\ref{foegyenlotlenseg}). 

Accordingly, we  assume that $D_M<\pi.$ Fix $x_0,\tilde x_0\in M$ such that $d_g(x_0,\tilde x_0)=D_M,$ and $y_0\in \mathbb S^n$. Let ${\rm d}v_g$ and ${\rm d}v_{g_0}$ be the canonical volume forms on $(M,g)$ and  $(\mathbb S^n,g_0)$, respectively. 
Let $f,s:(1,\infty)\to \mathbb R$ be the functions defined as   
\begin{equation}\label{ket-fuggveny}
f(\lambda)={\displaystyle \int_{M}(\lambda-\cos d_g)^{2-n}{\rm d}v_{g}}\ \ {\rm and}\ \  s(\lambda)={\displaystyle \int_{\mathbb{S}^{n}}(\lambda-\cos d_0)^{2-n}{\rm d}v_{\mathbb{S}^{n}},}\ \ \lambda>1,
\end{equation}
where $d_g=d_{g}(x_{0},\cdot)$ ans $d_0=d_{\mathbb S^n}(y_{0},\cdot)$. It is easily seen that both functions $f$  and $s$ are well-defined and smooth on $(1,\infty)$. 
 
 The proof will be provided in several steps. 
 
\noindent  \textbf{Step 1} (local behavior of $f$ and $s$ around $1$). We claim that
\begin{equation}\label{local-hatarertek}
\liminf_{\lambda\to1^{+}}\frac{f(\lambda)-\lambda f'(\lambda)}{s(\lambda)-\lambda s'(\lambda)}\geq 1.
\end{equation}
By the layer cake representation of functions and a change of variables, we have that 
\begin{eqnarray*}
I(\lambda)&:=& f(\lambda)-\lambda f'(\lambda)\\
&=&\int_M(\lambda-\cos d_g)^{1-n}((n-1)\lambda-\cos d_g){\rm d}v_{g}\\&=&
\int_0^\infty{\rm Vol}_g(\{x\in M:(\lambda-\cos d_g)^{1-n}((n-1)\lambda-\cos d_g)>t\}){\rm d}t
\\&=&
(n-2)\int_0^{D_M}{\rm Vol}_g(B_M({x_0},\rho))(\lambda-\cos \rho)^{-n}(n\lambda-\cos \rho)\sin \rho {\rm d}\rho\\ && + {\rm Vol}_g(M)(\lambda-\cos D_M)^{1-n}((n-1)\lambda-\cos D_M).
\end{eqnarray*}
In a similar manner, we have 
\begin{eqnarray*}
	J(\lambda)&:=& s(\lambda)-\lambda s'(\lambda)
	\\&=&
	(n-2)\int_0^{\pi}{\rm Vol}_{g_0}(B_{\mathbb S^n}({y_0},\rho))(\lambda-\cos \rho)^{-n}(n\lambda-\cos \rho)\sin \rho {\rm d}\rho\\ && + {\rm Vol}_{g_0}(\mathbb S^n)(\lambda+1)^{1-n}((n-1)\lambda+1).
\end{eqnarray*}
Fix $\varepsilon>0$ arbitrarily. Then the local behavior of the geodesic balls both on $(M,g)$ and $(\mathbb S^n,g_0)$ implies that there exits $\delta=\delta_\varepsilon>0$ sufficiently small such that for every $\rho\in (0,\delta)$, 
$${\rm Vol}_g(B_M({x_0},\rho))\geq (1-\varepsilon)\tilde \omega_n \rho^n
$$
and 
$${\rm Vol}_{g_0}(B_{\mathbb S^n}({y_0},\rho))\leq (1+\varepsilon)\tilde \omega_n \rho^n,$$
where $\tilde \omega_n$ denotes the volume of the $n$-dimensional unit ball in $\mathbb R^n$. 
Therefore, the above estimates give that  
\begin{equation}\label{IperJ}
\frac{I(\lambda)}{J(\lambda)}\geq \frac{(1-\varepsilon)(n-2)\tilde \omega_n\displaystyle\int_0^\delta (\lambda-\cos \rho)^{-n}(n\lambda-\cos \rho)\rho^n\sin\rho {\rm d}\rho}{(1+\varepsilon)(n-2)\tilde \omega_n\displaystyle\int_0^\delta (\lambda-\cos \rho)^{-n}(n\lambda-\cos \rho)\rho^n\sin\rho {\rm d}\rho +\tilde s(\lambda,\delta,n)},
\end{equation}
where 
\begin{eqnarray*}
\tilde s(\lambda,\delta,n)&=&(n-2)\int_\delta^{\pi}{\rm Vol}_{g_0}(B_{\mathbb S^n}({y_0},\rho))(\lambda-\cos \rho)^{-n}(n\lambda-\cos \rho)\sin \rho {\rm d}\rho \\&&+ {\rm Vol}_{g_0}(\mathbb S^n)(\lambda+1)^{1-n}((n-1)\lambda+1).
\end{eqnarray*}
Note first that $\tilde s(\lambda,\delta,n)=O(1)$ as $\lambda\to 1.$ Now, we show that 
\begin{equation}\label{vegtelen}
\lim_{\lambda\to 1}\int_0^\delta (\lambda-\cos \rho)^{-n}(n\lambda-\cos \rho)\rho^n\sin\rho {\textrm d}\rho=+\infty.
\end{equation}
Since $\cos\rho>1-\rho^2$, $n\lambda-\cos \rho\geq n-1$ and $\sin\rho\geq \frac{2}{\pi}\rho$ for every $\rho\in (0,\delta)$ and $\lambda>1$, it suffices to prove that 
$$
\lim_{\lambda\to 1}\int_0^\delta \frac{\rho^{n+1}}{(\lambda-1+ \rho^2)^{n}}  {\rm d}\rho=+\infty.
$$
In order to check the latter limit, by changes of variables  one has 
\begin{eqnarray*}
\lim_{\lambda\to 1}\left((\lambda-1)^{\frac{n}{2}-1}\int_0^\delta \frac{\rho^{n+1}}{(\lambda-1+ \rho^2)^{n}}  {\rm d}\rho\right)&= & \lim_{\lambda\to 1}\int_0^{\delta/\sqrt{\lambda-1}} \frac{\tau^{n+1}}{(1+ \tau^2)^{n}}  {\rm d}\tau \ \ \ \ \ \ \ \ [\rho=\sqrt{\lambda-1}\tau]\\
&=&\int_0^\infty \frac{\tau^{n+1}}{(1+ \tau^2)^{n}}{\rm d}\tau\\&=& \frac{1}{2}\int_0^1 \theta^\frac{n}{2}(1-\theta)^{\frac{n}{2}-2}{\rm d}\theta\ \ \ \ \ \ \ \ \ \ \  \left[\tau=\sqrt{\frac{\theta}{1-\theta}}\right]\\&=&
\frac{1}{2}{\textsf{ B}}\left(\frac{n}{2}+1,\frac{n}{2}-1\right).
\end{eqnarray*}

\noindent 
\textbf{Step 2} (ODE vs. ODI; global comparison of $f$ and $s$). 
Due
to Aubin \cite{Aubin}, the extremal function in (\ref{AB}) when $(M,g)=(\mathbb{S}^{n},g_0)$
is $u_{\lambda}(x)=(\lambda-\cos d_0)^{1-\frac{n}{2}}$ for every  $\lambda>1$. Thus, inserting $u_\lambda$ into (\ref{AB}) when $(M,g)=(\mathbb{S}^{n},g_0)$ and using the notation in (\ref{ket-fuggveny}), we have the following ODE:
\begin{equation}
\left[\frac{s''(\lambda)}{(n-2)(n-1)}\right]^{\frac{2}{2^{*}}}=\frac{2}{n}\omega_{n}^{-\frac{2}{n}}\left[\frac{1-\lambda^{2}}{2(n-1)}s''(\lambda)-\lambda s'(\lambda)+s(\lambda)\right],\ \lambda>1.\label{diffegyenlet}
\end{equation}

 Let $K_{0}=\frac{2}{n}\omega_{n}^{-\frac{2}{n}}$ and  $C=K_0\max\left\{\frac{A}{A_0},\frac{B}{B_0}\right\}$. Without loss of generality, we may assume that $A>A_0;$ indeed, since $A\geq A_0$, we may take $A=A_0+\varepsilon$ for $\varepsilon>0$ sufficiently small. Since $B\geq {\rm Vol}_{g}(M)^{-\frac{2}{n}} \geq B_0$, it turns out that  $C>K_0.$
By introducing the function
$$H(\lambda)=\left(\frac{K_{0}}{C}\right)^{\frac{n}{2}}J(s)=\left(\frac{K_{0}}{C}\right)^{\frac{n}{2}}(s(\lambda)-\lambda s'(\lambda)),$$
one has $H'(\lambda)=-\lambda\left(\frac{K_{0}}{C}\right)^{\frac{n}{2}}s''(\lambda),$
therefore $s''(\lambda)=-\frac{H'(\lambda)}{\lambda}\left(\frac{K_{0}}{C}\right)^{-\frac{n}{2}}$.
This means that the second order ODE (\ref{diffegyenlet}) is equivalent
to the following first order ODE:
\begin{equation}
\left[-\frac{H'(\lambda)}{\lambda(n-2)(n-1)}\right]^{\frac{2}{2^{*}}}=C\left[\frac{\lambda^{2}-1}{2\lambda(n-1)}H'(\lambda)+H(\lambda)\right], \ \lambda>1.\label{ode}
\end{equation}

Now, if we replace  $w_{\lambda}(x)=(\lambda-\cos d_g)^{1-\frac{n}{2}}$ for every $\lambda>1$ into (\ref{AB}) and we explore the eikonal equation $|\nabla_gd_g|=1$ valid a.e. on $M$, we obtain 
$$\left[\int_M (\lambda-\cos d_g)^{-n}{\rm d}v_g\right]^\frac{2}{2^*}\leq A \int_M (\lambda-\cos d_g)^{-n}\sin^2d_g{\rm d}v_g+B\int_M (\lambda-\cos d_g)^{2-n}{\rm d}v_g.$$
By using  the notation in (\ref{ket-fuggveny}),   the latter inequality can be rewritten into  
\[
\left[\frac{f''(\lambda)}{(n-2)(n-1)}\right]^{\frac{2}{2^{*}}}\leq K_0\left[\frac{A}{A_0}\frac{1-\lambda^{2}}{2(n-1)}f''(\lambda)-\frac{A}{A_0}{\lambda}f'(\lambda)+\left(\frac{2-n}{2}\frac{A}{A_0}+\frac{n}{2}\frac{B}{B_0}\right){f(\lambda)}\right],
\]
for every $\lambda>1.$
Since 
$$\frac{1-\lambda^{2}}{2(n-1)}f''(\lambda)-{\lambda}f'(\lambda)+\frac{2-n}{2}{f(\lambda)}=\frac{n-2}{2}\int_M(\lambda-\cos d_g)^{-n}\sin^2 d_g{\rm d}v_{g}\geq 0,$$
and 
 $C=K_0\max\left\{\frac{A}{A_0},\frac{B}{B_0}\right\}$, the latter inequality implies that
\[
\left[\frac{f''(\lambda)}{(n-2)(n-1)}\right]^{\frac{2}{2^{*}}}\leq C\left[\frac{1-\lambda^{2}}{2(n-1)}f''(\lambda)-{\lambda}f'(\lambda)+{f(\lambda)}\right],\ \ \lambda>1.
\]
Since $I(\lambda)=f(\lambda)-\lambda f'(\lambda)$, 
we get the following first order ordinary differential inequality:
\begin{equation}
\left[-\frac{I'(\lambda)}{\lambda(n-2)(n-1)}\right]^{\frac{2}{2^{*}}}\leq C\left[\frac{\lambda^{2}-1}{2\lambda(n-1)}I'(\lambda)+I(\lambda)\right].\label{odi}
\end{equation}

We claim that 
\begin{equation}\label{heswhasonlitas}
I(\lambda)\geq H(\lambda),\ \forall\lambda>1.
\end{equation}
First of all, by (\ref{local-hatarertek}) we clearly have  that 
\[
\liminf_{\lambda\to1^{+}}\frac{I(\lambda)}{H(\lambda)}=\liminf_{\lambda\to1^{+}}\frac{f(\lambda)-\lambda f'(\lambda)}{\left(\frac{K_{0}}{C}\right)^{\frac{n}{2}}(s(\lambda)-\lambda s'(\lambda))}\geq \left(\frac{C}{K_{0}}\right)^{\frac{n}{2}}>1.
\]
Thus, for sufficiently small $\delta_{0}>0$ one has 
\[
I(\lambda)\geq H(\lambda),\ \forall\lambda\in(1,\delta_{0}+1).
\]
 Assume by contradiction that $I(\lambda_{0})<H(\lambda_{0})$ for some $\lambda_{0}>1$.
Clearly, $\lambda_{0}>1+\delta_{0}.$ Let us define 
\[
\lambda_{s}:=\sup\{\lambda<\lambda_{0}:I(\lambda)=H(\lambda)\}<\lambda_{0}.
\]
Thus for any $\lambda\in[\lambda_{s},\lambda_{0}]$ we have $I(\lambda)\leq H(\lambda).$
It is also clear that 
\[
-\frac{I'(\lambda)}{\lambda(n-2)(n-1)}=\frac{f''(\lambda)}{(n-2)(n-1)}>0\]
 and
\[-\frac{H'(\lambda)}{\lambda(n-2)(n-1)}=\frac{s''(\lambda)}{(n-2)(n-1)}>0.
\]
Let us define the increasing function $\varphi_{\lambda}:(0,\infty)\to\mathbb{R}$
by 
\[
\varphi_{\lambda}(t)=t^\frac{2}{2^{*}}+\frac{(n-2)}{2}C(\lambda^{2}-1)t.
\]
By relations (\ref{ode}), (\ref{odi}) and the definition of  $\varphi_{\lambda}$, for every  $\lambda\in[\lambda_{s},\lambda_{0}]$ 
 we have that
\begin{align*}
\varphi_{\lambda}\left(-\frac{I'(\lambda)}{\lambda(n-2)(n-1)}\right) & =\left(-\frac{I'(\lambda)}{\lambda(n-2)(n-1)}\right)^{\frac{2}{2^{*}}}+\frac{(n-2)}{2}C(\lambda^{2}-1)\left(-\frac{I'(\lambda)}{\lambda(n-2)(n-1)}\right)\\
 & \leq CI(\lambda)\\
 & \leq CH(\lambda)\\
 & =\varphi_{\lambda}\left(-\frac{H'(\lambda)}{\lambda(n-2)(n-1)}\right).
\end{align*}
 Therefore, the  monotonicity of $\varphi_\lambda$ implies 
\[
I'(\lambda)\geq H'(\lambda),\ \forall\lambda\in[\lambda_{s},\lambda_{0}].
\]
In particular $\lambda\mapsto I(\lambda)-H(\lambda)$ is non-decreasing
on the interval $[\lambda_{s},\lambda_{0}].$ Consequently, we have
\[
0=I(\lambda_{s})-H(\lambda_{s})\leq I(\lambda_{0})-H(\lambda_{0})<0,
\]
  a contradiction, which shows the validity of (\ref{heswhasonlitas}).  
 
\noindent  {\bf Step 3} (proving (\ref{foegyenlotlenseg})).  
 Due to (\ref{Bishop-Gromov}), the claim is concluded once we prove
 \begin{equation}\label{vegso-egyenlotlenseg}
 \frac{{\rm Vol}_g(M)}{{\rm Vol}_{g_0}(\mathbb S^n)}\geq \min\left\{\frac{A_0}{A},\frac{B_0}{B}\right\}^\frac{n}{2}.
  \end{equation}
Note that relation  (\ref{heswhasonlitas}) is equivalent to
$$(n-2)\int_0^{D_M}{\rm Vol}_g(B_M({x_0},\rho))\frac{n\lambda-\cos \rho}{(\lambda-\cos \rho)^{n}}\sin \rho {\rm d}\rho + {\rm Vol}_g(M)\frac{(n-1)\lambda-\cos D_M}{(\lambda-\cos D_M)^{n-1}}$$
$$\geq \left(\frac{K_{0}}{C}\right)^{\frac{n}{2}} \left[(n-2)\int_0^{\pi}{\rm Vol}_{g_0}(B_{\mathbb S^n}({y_0},\rho)) \frac{n\lambda-\cos \rho}{(\lambda-\cos \rho)^{n}}\sin \rho {\rm d}\rho + {\rm Vol}_{g_0}(\mathbb S^n)\frac{(n-1)\lambda+1}{(\lambda+1)^{n-1}}\right], $$
for every $ \lambda>1.$

Let us multiply the above inequality by $\lambda^{n-2}$ and take the limit when $\lambda \to \infty$; the Lebesgue dominance theorem implies that both integrals tend to $0$, remaining 
$${\rm Vol}_g(M)\geq \left(\frac{K_{0}}{C}\right)^{\frac{n}{2}}{\rm Vol}_{g_0}(\mathbb S^n).$$
Since $C=K_0\max\left\{\frac{A}{A_0},\frac{B}{B_0}\right\}$, the latter relation implies  (\ref{vegso-egyenlotlenseg}) at once, which concludes the proof of (\ref{foegyenlotlenseg}).  
\end{proof}

%
%
%

 \begin{proof}[Proof of Corollary \ref{kovetkezmeny}]
  Since 	$\max\left\{\frac{A}{A_0},\frac{B}{B_0}\right\}\leq (1-\delta_n)^{-\frac{2}{n}},$
 by the quantitative volume estimate (\ref{foegyenlotlenseg}) it follows that 
 $${\rm Vol}_g(M)\geq (1-\delta_n){\rm Vol}_{g_0}(\mathbb S^n).$$
 The statement follows by  Cheeger and Colding \cite{CC-97}.  
 
 If $(M,g)$ is isometric to $(\mathbb{S}^{n},g_0)$, it is clear that   $A=A_0$ and $B=B_0,$ due to Aubin \cite{Aubin}. Conversely, when $A=A_0$ and $B=B_0,$ we apply (\ref{foegyenlotlenseg}) and (\ref{Bishop-Gromov}) in order to obtain ${\rm Vol}_{g_0}(B_{\mathbb S^n}(y_0,\rho))= {\rm Vol}_g(B_M(x_0,\rho))$ for every $\rho\in [0,\pi]$ (in fact, for every $\rho\in [0,\infty))$. Now, the equality in the Bishop-Gromov comparison principle implies that $(M,g)$ is isometric to $(\mathbb{S}^{n},g_0)$. \end{proof}
 
 \vspace{0.5cm}
 
 \textbf{Acknowledgment.}  {The authors are supported by the National Research, Development and Innovation Fund of Hungary, financed under the K$\_$18 funding scheme, Project No.  127926. A. Krist\'aly is also supported by the STAR-UBB grant. }

\end{document}